\documentclass[12pt,a4paper,english]{amsart}
\usepackage{graphicx}
\usepackage[left=20mm, right=20mm, top=25mm, bottom=25mm]{geometry}
\usepackage{color}
\usepackage{mathrsfs}
\usepackage{indentfirst, amsfonts, amsmath, amsthm, amscd,amssymb}
\usepackage{epsfig}
\usepackage{hyperref}
\usepackage{amsmath}
\usepackage{amsthm}
\usepackage{amssymb}
\usepackage{amscd}
\usepackage{amsfonts}

\newcommand{\R}{\mathbb{R}} 
\newcommand{\C}{\mathbb{C}}
\newcommand{\N}{\mathbb{N}} 
 
\newtheorem{theorem}{{Teorema}}[section]
\newtheorem{theo}[theorem]{{Theorem}}
\newtheorem{lema}[theorem]{{Lemma}}
\newtheorem{corol}[theorem]{{Corollary}}
\newtheorem{obs}[theorem]{{Remark}}
\newtheorem{defin}[theorem]{{Definition}}
\newtheorem{prop}[theorem]{{Proposition}}

\title[A nonlocal Chafee-Infante Problem without uniqueness]{Existence, regularity and asymptotic behavior of solutions for a nonlocal Chafee-Infante Problem via semigroup theory}
\author{Tom\'{a}s Caraballo, A. N. Carvalho, Yessica Julio }
\address[TC]{Depto. Ecuaciones Diferenciales y Anal. Num.\\
Facultad de Matem\'{a}ticas\\
Universidad de Sevilla\\
C/ Tarfia s/n\\
41012-Sevilla (Spain)}
\email[TC]{caraball@us.es}
\address[ANC and YJ]{
Instituto de Ci\^{e}ncias Ma\-te\-m\'{a}\-ti\-cas e de Computa\c{c}\~{a}o 
Universidade de S\~{a}o Paulo, Campus de S\~{a}o Carlos, Caixa Postal 668, S\~{a}o Carlos SP, Brazil.}
\email[ANC]{andcarva@icmc.usp.br}
\email[YJ]{yessica.julio@usp.br}
\thanks{[TC] Partially supported by the Spanish Ministerio de Ciencia e Innovaci{\'o}n (MCI), Agencia Estatal de Investigaci{\'o}n (AEI) and Fondo Europeo de Desarrollo Regional (FEDER) under the project PID2021-122991NB-C21.}
\thanks{[ANC] Partially supported by FAPESP Grant \# 20/14075-6 and by CNPq Grant \# 308902/2023-8, Brazil}
\thanks{[YJ] Partially supported  by CAPES Grant \# 88887.695331/2022-00 and  by the Colombian Ministerio de Ciencia, Tecnolog\'{\i}a e Innovaci\'{o}n (Minciencias).}
\begin{document}
\begin{abstract}
This article deals with the study of a non-local one-dimensional quasilinear problem with continuous forcing. We use a time-reparameterization to obtain a semilinear problem and study a more general equation using semigroup theory. The existence of mild solutions is established without uniqueness 
with the aid of the formula of variation of constants and asking only a suitable modulus of continuity on the nonlinearity this mild solution is shown to be strong. Comparison results are also established with the aid of the formula of variation of constants and using these comparison results, global existence is obtained with the additional requirement that the nonlinearity satisfy a structural condition. The existence of pullback attractor is also established for the associated multivalued process along with the uniform bounds given by the comparison results with the additional requirement that the nonlinearity be dissipative. As much as possible the results are abstract so that they can be also applied to other models.
\end{abstract}
\maketitle
\thispagestyle{empty} 
\section{Introduction}
Recently there has been much interest in the study of reaction-diffusion equations with non-local terms (see, e.g., \cite{chipot2003remarks,michel2001asymptotic, chipot1999asymptotic}).  Several developments have been achieved relatively to the associated elliptic problem and many interesting new features have appeared with respect to the local case. This makes the study of these problems interesting and challenging. \\
These developments can lead to a better understanding of the asymptotics of the solutions of the model and that can lead to better predictions of the behavior of the modeled phenomenon and allow for implementation of control strategies.
Of course, to understand the asymptotics of a given model, one must first tackle the problem of ensuring that solutions exist and that may become a quite challenging task requiring the usage of, sometimes challenging, methods (such as Faedo - Galerkin) or theories (such as semigroup theory and fixed point theorems) of mathematical analysis.
The prototype of the problem we want to study is an initial value problem of the form:
\begin{equation}\label{problema}
\left\{ \begin{array}{lcc}
\dfrac{\partial w}{\partial \tau} -a(l(w)) \dfrac{\partial w^2}{\partial x^2}=\lambda f(w) + h(\tau), \,\,\,\,\,\,\,\,\,\, \thinspace \tau>\sigma, \,\, x \in \Omega,  \\
\\ w(\tau,0)=w(\tau,1)=0,\\
\\ w(\sigma,x)=w_0(x),
\end{array}
\right.
\end{equation}
\noindent where $\Omega=(0,1),\:\lambda > 0$, $w_0\in H^1_0(0,1)$ , $l$ an operator from $H^1_0(0,1)$ to $\mathbb{R}$ and $a$, $f$ and $h$ are continuous functions satisfying some additional conditions depending on what are the properties that we will require for the solutions.
This equation has been investigated in \cite{caballero2021existence,chang2003nonlinear, chipot1999asymptotic, chipot2003asymptotic, michel2001asymptotic}, from different perspectives and appeared in the literature for the first time in $1974$ (see \cite{chafee-infante,chafee1974bifurcation}).
The existence of a solution of a differential equation can be shown in different ways, in the case of our prototype problem \eqref{problema} the authors in \cite{ruben} proved that, if $f \in C(\R)$, $a \in C(\R^+)$ is such that $a(s)\geqslant m > 0$, $h \in L^2_{loc}(0,+ \infty; L^2(\Omega))$, and if there are constants $\beta, \gamma > 0$ such that $f(u)u \leqslant \gamma + \beta u^2$ for all $u \in \R,$ the initial value problem has a solution, for which the authors used the Faedo-Galerkin method, which  consists of building a solution of the problem from the construction of solutions of certain finite-dimensional approximations of the original problem and then passing to the limit.\\
We observe that the condition $f(u)u \leqslant \gamma + \beta u^2$, for all $u \in \R$, is not a natural condition for local existence of solutions. It is associated to the method used to construct the solutions, to ensure that all approximating solutions exist in a common interval of existence (in particular, to the fact that solutions are defined for all $t\geqslant 0$).  We wish to give results on existence of solutions which do not require such type of conditions. \\
Our approach will require, for existence of mild solutions (continuous in time functions taking value in a suitable phase space), only that $a$ is bounded away from zero and the continuity of $f$, $a\circ l$, and $h$.  
\bigskip
To obtain that all solutions are global we will need to impose a structural condition
\noindent (\textbf{S}) Assume that there exist $C_0,C_1 \in \R$ such that 

$$
uf(u) \leqslant -\nu C_0u^2 + |u|C_1
$$

for all $u\in \mathbb{R}$ and for both $\nu=\frac{m}{\lambda}$ and $\nu=\frac{M}{\lambda}$.

\bigskip

To prove the existence of a pullback attractor we assume, in addition to (\textbf{S}), a dissipativity condition on $f$. This condition is expressed as: 

\bigskip

\noindent (\textbf{D}) Assume that (\textbf{S}) holds for some $C_0$ such that the first eigenvalue $\lambda_1$ of $A+\nu C_0I$ is positive.

\bigskip

We also tackle the regularity of solutions. In fact, for the existence of a classical solution, we will also need to require that $a\circ l$, $f$ and $h$ satisfy some suitable conditions of integrability for the modulus of continuity in bounded sets. \\
Once the local existence is established for all initial conditions in a suitable space, we obtain global existence of solutions for all initial conditions assuming some almost monotonicity condition (as that of  \cite{ruben}) on the reaction term $f$.\\
Our approach is abstract so that the results may be applied to many other models under suitable conditions. The results are then specialized to the model in \eqref{problema}.\\
The paper is divided as follows. In Section \ref{setting_of_the_problem}, we introduce some essential definitions and precisely state the main results.  In Section \ref{existence_of_solutions}, we prove the results of Section \ref{setting_of_the_problem}, that is,  the existence of mild and classical solutions (without uniqueness). In Section 4, we provide some results about the comparison of solutions. Finally, in Section 5, we recall the notion of multivalued evolution processes to show the existence of a pullback attractor.

\section{Setting of the problem and statement of the results}\label{setting_of_the_problem}

Given a function $u$ that solves the problem

\begin{equation}\label{problemau}
\left\{ \begin{array}{lcc}
u_t = u_{xx} + \dfrac{\lambda f(u) + h(\alpha(t))}{a(l(u))}, \,\,\,\,\,\,\,\,\,\, \thinspace t >\sigma, \,\, x \in \Omega,  \\
\\ u(t,0)=u(t,1)=0,\\
\\ u(\sigma,x)=w_0(x),
\end{array}
\right.
\end{equation}

where $\alpha(t)=:\sigma+\int_{\sigma}^t  a(l(u(r)))dr$. Define $\tau=\alpha(t)$ and $w(\tau)=u(t)$.
Then $w(\tau)$ is a solution of the non-local initial value problem \eqref{problema}. The inverse $\alpha^{-1}:[\sigma,\infty)\to [\sigma,\infty)$ of the function $\alpha:[\sigma,\infty)\to [\sigma,\infty)$ is the function
$$
\alpha^{-1}(\tau)= \sigma +\int_\sigma^\tau a(l(w(r)))^{-1}dr
$$

On the other hand, given a function $w$ that solves  \eqref{problema} making $t\!:=\!\alpha^{-1}\!(\tau)\!=\!\sigma+\!\int_\sigma^\tau \!a(l(w(r)))^{-1}\! dr$ we have that it also solves

\begin{equation}\label{problema'}
\left\{ \begin{array}{lcc}
\dfrac{\partial w}{\partial \tau} -a(l(w)) \dfrac{\partial w^2}{\partial x^2}=\lambda f(w) + h\circ\alpha\circ\alpha^{-1}(\tau), \,\,\,\,\,\,\,\,\,\, \thinspace \tau>\sigma, \,\, x \in \Omega,  \\
\\ w(\tau,0)=w(\tau,1)=0,\\
\\ w(\sigma,x)=w_0(x),
\end{array}
\right.
\end{equation}
and the function $u$ defined by $u(t)=w(\tau)$ satisfies \eqref{problemau}.

\bigskip

Having established these relations between the regular solutions of \eqref{problema} and \eqref{problemau}, we will prove a well posedness result for an abstract model that includes the well posedness of \eqref{problemau} and that will give us also the well posedness of \eqref{problema}.\\
One must note that non only problem \eqref{problemau} remains non-local in space, but now it is also nonlocal in time since the term $h(\alpha(t))$ depends on the solution from the initial time $\sigma$ until the final time $t$. We also note that the change in the time scale, that transformed problem \eqref{problema} into \eqref{problemau}, depends on the specific solution under consideration and therefore it is not a change of variable but a way to transform the problem of finding a solution for the quasilinear problem \eqref{problema} into the simpler problem of finding a solution for semilinear problem \eqref{problemau} which is non-local in time and in space. \\
This method introduces considerable advantages with respect to treating the original problem \eqref{problema} with techniques of quasilinear problems, in particular with respect to techniques involving the use of semigroup theory (see \cite{Lunardi,Yagi}). More specifically, the regularity requirements on the non-local diffusion are much lower when we consider the transformed problem \eqref{problemau}. With respect to the Faedo-Galerkin method, our method is complementary and treats the local existence and global existence differently with different conditions and allows for the usage of the variation of constants formula which is, many times, quite convenient for the study of the asymptotics.

\bigskip

Let $X$ be a Banach space and denote by $L(X)$ the space of bounded linear operators and $C(X)$ the space of continuous functions from X into itself.
Consider the problem
\begin{equation}\label{general}
\left\{ \begin{array}{lcc}
\dfrac{ du(t)}{dt}+  Au(t)= f(t,u^{t}(t),u^t(\cdot))  \,\,\,\,\,\,\,\,\,\, \thinspace t>\sigma, \,\, \\
\\ u(\sigma)=w_0 \in X^\alpha,
\end{array}
\right.
\end{equation}
where $-A$ has compact resolvent and is the infinitesimal generator of an exponentially decaying analytic semigroup $\{e^{-At}:t\geqslant 0\}$, $X^\alpha, \: \alpha \geqslant 0$ is the fractional power spaces associated to $A$ \cite{komatsu1966fractional,henry1981geometric}, and $f:[\sigma,\sigma+T] \times X^\alpha \times C([\sigma,\sigma+T], X^\alpha) \rightarrow X$ is a continuous function.
Here, for $u(\cdot) \in C([\sigma,\sigma+T], X^\alpha)$, $u^{t}(\cdot) = u(\cdot)|_{{[\sigma,t]}}$, $u(t)=u^{t}(t)$. 

\bigskip

For our specific problem, for each $\sigma$, we have 
\begin{equation}\label{gdef}
g(t,u^{t}(t),u^t(\cdot))=\frac{1}{a(l(u(t)))}\left[\lambda f(u(t)) + h(\sigma+ \displaystyle \int_\sigma^t a(l(u(r)))dr)\right].    
\end{equation}
With the purpose of ensuring the existence of  classical solution for \eqref{general}, we will need to impose some additional restrictions on the nonlinear forcing term $f$.\\
One can prove, using arguments similar to those used in \cite[Teorema 6.2.1]{pazy}, that if $f:[\sigma,\sigma+T) \times X^\alpha \times  C([\sigma,\sigma+T], X^\alpha) \rightarrow X$ is continuous, the initial value problem \eqref{general} has a mild solution $u \in C([\sigma,\sigma+T];X^\alpha)$, and that if $f$ is locally H\"{o}lder continuous in $t$ and locally Lipschitzian in $u$, one can guarantee the existence of a unique classical solution (See Theorem 3.3.3 in \cite{henry1981geometric}).\\
Following this reasoning, we prove existence of a mild solution to the problem \eqref{general} for $f:[\sigma,\sigma+T) \times X^\alpha \times  C([\sigma,\sigma+T], X^\alpha) \rightarrow X$ being only continuous, and show the existence of a strong solution for $f:[\sigma,\sigma+T) \times X^\alpha \times  C([\sigma,\sigma+T], X^\alpha) \to X$ being such that:
 Given $u\in C([\sigma,\sigma+T], X^{\alpha})$ 
 
 $$
\|f(t,u^t(t),u^t(\cdot))-f(s,u^s(s),u^s(\cdot))\|_X \leqslant w(\|u(t)-u(s)\|_{\alpha})+w(|t-s|^\beta),\,\,0 <\beta<1-\alpha,
$$

with $w:[0,\infty) \rightarrow [0,\infty)$ being a continuous increasing function, $w(0)=0$ and

$$
   \displaystyle \int_{0}^{t}u^{-1}w(u^{\beta})du < \infty.
$$ 

In our specific case, to ensure the existence of mild solution, it is only necessary $f(\cdot),~h(\cdot)$ and $a(l(\cdot))$  being continuous, and for the existence of strong solution (which is classical by simple bootstrapping arguments) we need that $g:[\sigma,\sigma+T) \times X^\alpha \times  C([\sigma,\sigma+T], X^\alpha) \rightarrow X$ satisfies the following: \\
\begin{itemize}
\item $a$ being continuous and $a(r)\geqslant m > 0$, for all $r\geqslant 0$;
\item $l$ a continuous operator from $X^\alpha$ to $\mathbb{R}^+$; 
\item given $u_0,v_0\in X^\alpha$
\begin{equation}\label{modulusofcont}
\begin{split}
&\| f(u_0)-f(v_0) \|_X \leqslant w(\| u_0-v_0\|_\alpha) \\
&\left\| a(l(u_0))-a(l(v_0)) \right\| \leqslant w(\| u_0-v_0\|_\alpha) ,
\\
&\|h(t)-h(s)\| \leqslant w(|t-s|^\beta),
\end{split}
\end{equation}
where $w:[0,\infty) \rightarrow [0,\infty)$ is a continuous increasing function, $w(0)=0$ and
$$
 \displaystyle \int_{0}^{t}u^{-1}w(u^{\beta})du < \infty,
\hbox{ for some }\beta\in (0,1-\alpha).
$$ 
\end{itemize}

Let us now introduce some more terminology in order to define the solutions that we will be dealing with. We start with the definition of sectorial operators.

\begin{defin}\cite{george}
Let $X$ be a Banach space. A linear operator $A:D(A)\subset X\to X $ is a sectorial operator if it is a closed densely defined operator such that, for some $\phi$ in $(0,\pi/2)$ and some $M \geqslant 1$ and real number $a$, the sector 
\[
\Sigma_{a,\phi} = \{ \lambda \in \C: \phi \leqslant  |arg(\lambda-a)| \leqslant \pi, \; \lambda \neq a \} 
\]
is in the resolvent set of $A$ and 
\[
||(\lambda-A)^{-1}|| \leqslant M/|\lambda-a| \;\;for\;\;all\;\;\lambda \in \Sigma_{a,\phi}.
\]
\end{defin}

Next we present the definition of analytic semigroup.

\begin{defin}\cite{pazy}
Let $\Delta = \{z \in \C: \:\phi_1 < arg\: z < \phi_2,\: \phi_1<0<\phi_2\}$ and for $z \in \Delta  $ let $T(z)$ be a bounded linear operator. The family $\{T(z), \:z \in \Delta \}$ is an analytic semigroup in $\Delta$ if
\begin{itemize}
\item [i.] $z \rightarrow T(z)$ is analytic in  $\Delta.$
\item [ii.] $T(0) = I$ and $\lim\limits_{
z \to 0,\:z \in \Delta
} T(z)x = x $ for all $x \in X^\alpha.$
\item[iii.] $T(z_1 + z_2) = T(z_1)T(z_2)$ for $z_1,\:z_2 \in \Delta.$
\end{itemize}
\end{defin}

A sectorial operator $A$ is such that $-A$ is the generator of an analytic semigroup. 
With this we introduce the definitions of strong and mild solution for \eqref{general}.

\begin{defin}\cite{pazy}
A function $u:[\sigma,\sigma+T) \rightarrow X$ is a strong solution of \eqref{general} on $[\sigma,\sigma+T),$ if $u$ is continuous on $[\sigma,\sigma+T)$ and continuously differentiable on $(\sigma,\sigma+T), \: u(t) \in D(A)$ for $\sigma<t<\sigma+T,~u(\sigma)=w_0$  and \eqref{general} is satisfied on $[\sigma,\sigma+T).$
\end{defin}

\begin{defin}\cite{pazy}
Let $A$ be the infinitesimal generator of an analytic semigroup $T(t)$, $w_0 \in X^\alpha $ and $f :[\sigma,\sigma+T) \times X^\alpha \times  C([\sigma,\sigma+T], X^\alpha) \rightarrow X$.
The function $ u \in C([\sigma,\sigma+T]; X^\alpha) $ such that
\begin{equation}\label{solufraca}
u(t)=T(t-\sigma)w_0 + \int_\sigma^t T(t-s)f(s,u^s(s),u^s(\cdot)) ds, \qquad \sigma \leqslant t \leqslant \sigma+T,
\end{equation}
is called a mild solution of the initial value problem \eqref{general} on $[\sigma,\sigma+T].$
\end{defin}

In \cite[Theorem 6.2.1]{pazy}  the existence of a mild solution for \eqref{general} when $-A$ is the infinitesimal generator of a compact $C_0$ semigroup was proved ($\alpha=0$). We note that often, when passing to applications, generators of compact semigroups occur when $-A$ has a compact resolvent and generates an analytic semigroup $\{T(t): t \geqslant 0\}$.\\
We extend the theory in \cite[Theorem 6.2.1]{pazy} to the case when $-A$ is the infinitesimal generator of a compact analytic semigroup.\\
In Section \ref{existence_of_solutions} we prove  the existence of solutions and their regularity. In fact we prove the following result concerning the existence of solutions.

\begin{theo}\label{mildsolution}
Let $X$ be a Banach space, $A$ a sectorial operator such that $(\lambda + A)^{-1}$ is compact for all $\lambda \in \rho (-A)$, and let $\{T(t): \,t\geqslant 0\}$ be the analytic semigroup generated by $-A$. If $f:[\sigma,\sigma+T]\times X^\alpha \times C([\sigma,\sigma+T],X^\alpha) \rightarrow X, \, 0\leqslant \alpha \leqslant 1$ is a continuous map
then for each  $w_0\in X^\alpha$ there exists a $T_1=T_1(w_0)\in (0,T]$ such that the initial value problem \eqref{general} has a mild solution $u \in C([\sigma,\sigma+T_1];X^\alpha)$. Furthermore, $T_1$  may be chosen uniformly for $w_0$ in bounded subsets of $X^\alpha$.

\end{theo}

We note that, as a consequence of Theorem \ref{mildsolution}, if $u(\cdot,w_0)$ is defined in a finite time interval of existence $[0,\tau)$ and $\sup_{t\in [0,\tau)}\|u(t,w_0)\|_\alpha<\infty$ then $u(\cdot,w_0)$ can be extended to an interval $[0,\tau')$ with $\tau'>\tau$. We also prove that, under suitable conditions, strong solutions are classical.

\begin{theo}\label{strongsolution} Let $X$ be a Banach space, $A$ a sectorial operator such that $(\lambda + A)^{-1}$ is compact for all $\lambda \in \rho (-A)$, and $\{T(t): \,t\geqslant 0\}$ be an analytic semigroup generated by $-A$. Assume that $f: [\sigma,\sigma+T) \times X^\alpha \times  C([\sigma,\sigma+T], X^{\alpha}) \rightarrow X, \, 0\leqslant \alpha \leqslant 1$, is continuous and such that, given $u \in C([\sigma,\sigma+T], X^{\alpha}) $ 

\begin{equation}\label{hipsobref}
\|f(t,u^t(t),u^t(\cdot))-f(s,u^s(s),u^s(\cdot))\|_X \leqslant w(\|u(t)-u(s)\|_{\alpha})+w(|t-s|^\beta),\,\,0 <\beta<1-\alpha,
\end{equation}
where $w:[0,\infty) \rightarrow [0,\infty)$ is an increasing continuous function such that $w(0)=0$ and 
\begin{equation}\label{hipsobrew}
\displaystyle \int_{0}^{t}u^{-1}w(u^\beta)du < \infty .
\end{equation}

If $u:[\sigma,\sigma+t_1] \rightarrow X^{\alpha}$ is  a continuous function and
$$
u(t)=T(t-\sigma)w_0 + \displaystyle \int_{\sigma}^{t} T(t-s) f(s,u^s(s),u^s(\cdot))ds , \;\;\ \; \sigma \leqslant t \leqslant \sigma+t_1,
$$

then $u(t)$ is a strong solution of \eqref{general}.
\end{theo}

We also find suitable conditions to apply these results to our prototype model \eqref{problemau}.

Section \ref{comparison_results} is dedicated to the establishment of comparison results for solutions of problems like our prototype problem \eqref{problemau}. To that end we extend the results obtained in \cite{acrb-2000}. This is aimed towards obtaining global solutions without strong restrictions on the nonlinearities and also to obtain bounds on the attractors for the resulting multivalued flows.

\section{Existence of Solutions}\label{existence_of_solutions}
In this section, we shall establish the existence of a mild solution and strong solution for problem \eqref{general}. We will follow along the same lines as a similar result in \cite{pazy} proving the existence of a (possibly non-unique) mild solution. 

\begin{proof}[Proof of Theorem \ref{mildsolution}]
For simplicity of notation and without loss of generality, we will only consider the case when the initial time $\sigma=0$. For the general case, simply consider a new time variable $t'=t-\sigma$. We know $\{T(t): \,t\geqslant 0\}$ is a compact exponentially decaying analytic semigroup. Then (see \cite[Theorem 1.4.3]{henry1981geometric})
\begin{equation}\label{asg_est}
\| A^{\alpha} T(t) \| \leqslant M_\alpha t^{-\alpha}e^{-\omega t},\ \ \ t>0,
\end{equation}
and, for $0<\alpha\leq 1$ and $x\in X^\alpha$,
\begin{equation}\label{asg_holder}
\| T(t)x-x \| \leqslant \frac{M_{1-\alpha}}{\alpha} t^{\alpha}\|x\|_\alpha,\ \ \ t>0.
\end{equation}

For $w_0\in X^\alpha$,  $a\in (0,T]$ and $\rho>0$ define
$$
B_{\rho\!,a}(w_0):=\{v \in C([0,T],X^\alpha): \, \| v(t)-w_0\|_\alpha\leqslant \rho, 0 \leqslant t \leqslant a\}.
$$
and
$$ 
N_a=\sup\left\{\| f(t,v^t(t),v^t(\cdot) )\|_X: v \in B_{\rho,a}(w_0),\,\,0 \leqslant t \leqslant a\right\}.
$$  
Let $a>0$ be such that
$$
\sup\{\| T(t)w_0-w_0 \|_{\alpha}:t\in [0,a]\} \leqslant \rho/2,
$$
$$
N_aM_\alpha \displaystyle \int_0^a u^{-\alpha}e^{-\omega u}du \leqslant \rho/2
$$
and choose
$$
T_{1}=\min \{ a, \rho/2M_{\alpha}N_a \}.
$$
Set $Y= C([0,T_1];X^\alpha)$ and 
$
Y_{0}=\{ u: \,u \in B_{\rho,T_1}(w_0) , \, u(0
)=w_0\}.
$
Note that $Y_0$ is a bounded closed convex subset of $Y$. For $u\in Y$ we define 
$$
(Fu)(t)=T(t)w_0 \; + \; \int_{0}^{t}  T(t-s) f(s,u^s(s),u^s(\cdot)) ds, \ t\in [0,T_1].
$$
It is easy to see that $F$ maps $Y$ into itself and that $(Fu)(0)=w_0$. Since, for $u\in Y_0$,
\[
\begin{array}{rl}
\| (Fu)(t)-w_0 \|_{\alpha}& \leqslant  
\| T(t)w_0-w_0  \|_{\alpha} \; + \displaystyle \int_{0}^{t} \| T(t-s) f(s,u^s(s),u^s(\cdot))\|_{\alpha} ds.\\\\
& \leqslant \rho/2 \; + \displaystyle \int_{0}^{t} \| A^\alpha T(t-s)\|_{L(X)} \| f(s,u^s(s),u^s(\cdot))\|_{X} ds.\\\\
& \leqslant \rho/2 + N_aM_\alpha \displaystyle \int_0^t (t-s)^{-\alpha}e^{-\omega (t-s)}ds \; \leqslant \; \rho,\quad  t \in [0,T_1],
\end{array}
\]
we have that $F$ maps $Y_0$ into $Y_0$. Also, thanks to the continuity of $f$ it follows that the map $F$ is continuous. 

We will prove that $\overline{F(Y_0)}$ is compact in $Y_0$. To that end we first show that for every fixed $t, \; 0 \leqslant t \leqslant T_{1}$, the set $Z_0(t)= \{(Fu)(t), \, u \in Y_0 \}$ is precompact in $X^\alpha$. 

This is clear for $t=0$. Now, for $t>0$ and $\beta >0$ such that $0 \leqslant \alpha+ \beta \leqslant 1$ we have
\[
\begin{array}{rl}
\| A^{\beta}(Fu)(t) \|_{\alpha} & \leqslant \| A^{\beta}T(t)w_0 \|_{\alpha} + \displaystyle \int_{0}^{t} \| A^{\beta} T(t-s) f(s,u^s(s),u^s(\cdot))\|_{\alpha} ds  \\\\
& \leqslant M_{\beta} t^{ -\beta} e^{-\omega t}\|w_0\|_{\alpha} + N_aM_{\alpha+\beta} \displaystyle \int_0^t (t-s)^{-\alpha-\beta}e^{-\omega(t-s)}ds  < \infty
\end{array}
\]
since $A$ has compact resolvent, $X^{\alpha + \beta }$ is compactly embedded in $X^{\alpha}$  and it follows that $Z_0(t)$ is precompact in $X^\alpha$.\\

Now we are going to prove that
$
F(Y_0) = \tilde{Y} = \{ Fu : \; u \in Y_0 \}
$
is an equicontinuous family of functions. For $t_2>t_1 \geqslant 0$ we have 

\begin{equation*}
\begin{split}
\| (Fu)&(t_2) - (Fu)(t_1) \|_{\alpha}  \leqslant \| (T(t_2)-T(t_1)w_0)\|_{\alpha}\\
& + \| \displaystyle \int_0^{t_2} T(t_{2}-s)f(s,u^s(s),u^s(\cdot))ds- \int_0^{t_1}T(t_2-s)f(s,u^s(s),u^s(\cdot)) ds\|_{\alpha}\\
& + \|\int_0^{t_1}T(t_2-s)f(s,u^s(s),u^s(\cdot)) ds    - \displaystyle \int_0^{t_1} T(t_{1}-s)f(s,u^s(s),u^s(\cdot))ds  \|_{\alpha}\\
&  \leqslant \| (T(t_1)(T(t_2-t_1)-I)A^\alpha w_0\| \\
& +N_a \displaystyle \int_0^{t_1} \|T(t_{2}-s) -  T(t_{1}-s) \|_{L(X,X^\alpha)} ds  + \displaystyle N_a \int_{t_1}^{t_2}  \|T(t_{2}-s)\|_{L(X,X^\alpha)}  ds \\
&\leqslant  M_0e^{-\omega t_1}\|(T(t_2-t_1)-I)A^\alpha w_0)\|\\ 
& + N_a\!\! \displaystyle \int_0^{t_1} \!\!\!\|(T(t_2-t_1)\!-\!I) A^{\!-\beta}\|_{L(X)} \|A^{\beta+ \alpha }T( t_1-s )\|_{L(X)}ds 
 + N_aM_\alpha \!\!\displaystyle \int_{t_1}^{t_2}\!\! (t_2 - s) ^{-\alpha}\! e^{-\omega (t_{2}-s)}\!ds\\
&\leqslant M_0 e^{-\omega t_1}\| (T(t_2-t_1)-I)A^\alpha w_0\| 
 \\&+  \frac{N_aM_{1-\beta}M_{\alpha+\beta}}{\beta} (t_2 - t_1 )^{\beta}  \displaystyle \int_0^{t_1} (t_1 -s)^{-\alpha-\beta} e^{-\omega (t_1 - s)} ds 
 + N_aM_\alpha \displaystyle \int_{t_1}^{t_2} (t_2 - s) ^{-\alpha} e^{-\omega (t_{2}-s)}ds\\
 &\leqslant M_0 \| (T(t_2-t_1)-I)A^\alpha w_0\| 
 \\&+  \frac{N_aM_{1-\beta}M_{\alpha+\beta}}{\beta} (t_2 - t_1 )^{\beta}  \displaystyle \int_0^{t_1} s^{-\alpha-\beta} e^{-\omega s} ds 
 + N_aM_\alpha \displaystyle \int_{t_1}^{t_2} (t_2 - s) ^{-\alpha} e^{-\omega (t_{2}-s)}ds\\
  &\leqslant M_0\| (T(t_2-t_1)\!-\!I)A^\alpha w_0\| +  \frac{N_aM_{1-\beta}M_{\alpha+\beta}}{\beta} (t_2\! -\! t_1 )^{\beta} \frac{{t_1}^{1-\alpha-\beta}}{1-\alpha-\beta}
 +  \frac{N_aM_\alpha}{1-\alpha}(t_2-t_1)^{1-\alpha}
\end{split}
\end{equation*}
this last expression is independent of $u \in Y_0$ and tends to zero as $t_2 - t_1 \rightarrow 0$.
Note also that 
\[
\sup_{t \in [0,T_1], \; u \in Y_0} \|(Fu)(t)\|_\alpha = \sup_{t \in [0,T_1], \; u \in Y_0} \|A^\alpha(Fu)(t)\| < \infty.
\]
The desired precompactness of $\tilde{Y}$ is now a consequence of Arzela-Ascoli's theorem. It follows from  Shauder's fixed point theorem that $F$ has a fixed point in $Y_0$ which is a mild solution of \eqref{general}.
\end{proof}

If we assume further that $f$ satisfies \eqref{hipsobref}, we will be able to obtain a strong solution to the initial value problem \eqref{general}. In general, such regularity results require that $f$ be H\"{o}lder continuous in time and Lipschitz continuous in the phase space \cite{pazy,henry1981geometric}. We improve those results requiring a suitable logarithmic the modulus of continuity for the function $f$.

\begin{proof}[Proof of Theorem \ref{strongsolution}]
Let $0<\beta< 1-\alpha$. We will prove first that for some $C>0$
$$
\|f(t,u^t(t),u^t(\cdot))-f(s,u^s(s),u^s(\cdot))\|_X \leqslant w ( C(\min\{t,s\})^{-\beta}|t-s|^{\beta} ) + w(|t-s|^\beta).
$$

Since $f$ is continuous 
\begin{equation}\label{estima_f}
\sup_{0 \leqslant t \leqslant T_1} \| f(t,u^{t}(t),u^t(\cdot)) \|_X \leqslant B.
\end{equation}
If $0 \leqslant t \leqslant t+h \leqslant  T_1$, we have that 
\begin{equation}
\begin{split}
u(t+h)-u(t) = & ( T(h)-I ) \left[ T(t)w_0 + \displaystyle \int_{0}^{t} T(t-s) f(s,u^{s}(s),u^s(\cdot))ds \right]\\\\
  & + \displaystyle \int_{t}^{t+h} T(t+h-s) f(s,u^{s}(s),u^s(\cdot)) ds
\end{split}
  \end{equation}
it follows that, if $0 <  \alpha < 1 $ and $ 0 < h < 1 $,   
then 
\begin{equation}\label{estima_u}
\begin{split}
\|u(t+&h)  -u(t)\|_{\alpha} = \left\| (T(h)-I ) \left[T(t)w_0  +  \int_{0}^{t} T(t-s) f(s,u^{s}(s),u^s(\cdot))ds  \right]\right\|_{\alpha}\\\\
& + \left\| \displaystyle \int_{t}^{t+h} T(t+h-s) f(s,u^{s}(s),u^s(\cdot)) ds \right\|_{\alpha} \\\\
& =  \|( T(h)-I ) A^{\alpha} T(t)w_0 \|_X +  \displaystyle  \int_{0}^{t} \| ( T(h)-I ) A^{\alpha}T(t-s) f(s,u^{s}(s),u^s(\cdot))\|_{X} ds \\\\
& + \displaystyle \int_{t}^{t+h} \| A^{\alpha}T(t+h-s) f(s,u^{s}(s),u^s(\cdot)) \|_{X} ds \\\\
& = I_1 + I_2 + I_3. 
\end{split}
\end{equation}

We note that, for every $\beta$ satisfying $ 0 < \beta < 1- \alpha $ and every $ 0 < h < 1 $, we have from \eqref{asg_holder}

\begin{equation}\label{estimativafundamental}
\|( T(h)-I ) A^\alpha T(r) \|_{L(X)} \leqslant \frac{M_{1-\beta}}{\beta} h^{\beta} \|A^{\alpha + \beta} T(r) \|_{L(X)} \leqslant  \frac{M_{1-\beta}M_{\alpha+\beta}}{\beta} h^{\beta}r^{-(\alpha + \beta)}, \ r>0.
\end{equation}

Using \eqref{asg_holder}, \eqref{estima_f} and (\ref{estimativafundamental}) we estimate each of the terms of (\ref{estima_u}) separately.\\
\begin{equation*}
\begin{split}
&I_1  \!=\! \|( T(h)-I )T(t)w_0\|_{\alpha} \leqslant \frac{M_{1-\beta}M_{\beta}}{\beta}h^{\beta} t^{-\beta}\|w_0\|_{\alpha},\\
&I_2 \!=\!\! \displaystyle  \int_{0}^{t}\!\! \| ( T(h)-I ) A^{\alpha}T(t-s) f(s,u^{s}(s),u^s(\cdot))\| ds \!\leqslant \!\!\displaystyle  \int_{0}^{t}\!\! B\frac{M_{1-\beta}M_{\alpha+\beta}}{\beta}h^{\beta} (t-s)^{-(\alpha + \beta)}  ds,\\
&  I_3  \!= \!\! \int_{t}^{t+h}\!\!\!\!\! \|\!\!A^{\alpha}T(t+h-s) f(s,u^{s}(s),u^s(\cdot)) \| ds    \!\leqslant \!\!  \int_{t}^{t+h} \!\!\!\!\!\!\!BM_\alpha(t+h-s)^{-\alpha}ds \!\leqslant\! \frac{BM_\alpha}{1-\alpha} h^{1-\alpha} \!\leqslant\! \frac{BM_\alpha}{1-\alpha} h^\beta.\\
\end{split}
\end{equation*}

Combining (\ref{estima_u}) with these estimates it follows that there exists a constant $C>0$ such that 
$$
\|u(t+h)-u(t)\|_{\alpha} \leqslant Ct^{-\beta}h^{\beta} \;\;\; for \;0 < t \leqslant T_1-h,
$$
and therefore 
\begin{equation*}
\begin{split}
 \| f(t,u^t(t),u^t(\cdot))-f(s,u^s(s),u^s(\cdot)) \|_X  &\leqslant w ( \|u(t)-u(s)\|_{\alpha} ) +w(|t-s|^\beta)\\
 &\leqslant w ( C(\min\{t,s\})^{-\beta}|t-s|^{\beta} )+w(|t-s|^\beta), 
 \end{split}
\end{equation*}
for $0 <  t,s \leqslant T_1 $ and $C, \, \beta > 0, \, 0 < \beta < 1 - \alpha.$\\
To prove that $u:[0,T_1] \rightarrow X^{\alpha}$ is a strong solution it is enough to prove that
$$
G(t) = \displaystyle  \int_{0}^{t} e^{-A(t-s)}f(s,u^s(s),u^s(\cdot))ds
$$
takes values into $D(A)$ and $t \longmapsto -AG(t)$ is continuous on $(0,T_1]$. For this we are going to prove that $h^{-1}(T(h)-I)G(t)$ converges as $h \rightarrow 0^+$ uniformly for $t\in [t_0^* , T_1]$, for each $t_0^*\in (0,T_1]$.\\
Indeed, we have
\begin{equation}
\begin{array}{rl}\label{AG(t)}
h^{-1} ( T(h)-I )G(t) &  = \displaystyle  \int_{0}^{t} h^{-1} ( T(h)-I )T(t-s)(f(s,u^s(s),u^s(\cdot))-f(t,u^t(t),u^t(\cdot)))ds \\\\
& + h^{-1} \displaystyle  \int_{0}^{t} ( T(h)-I )T(t-s)f(t,u^t(t),u^t(\cdot)))ds.
\end{array}
\end{equation}

For simplicity of notation we will write $g(r)$ for $f(r,u^r(r),u^r(\cdot))$. To prove the convergence of the first integral in (\ref{AG(t)}) we are going to use Lebesgue's Dominated Convergence Theorem.
To that end, first, we note that, for $0<\delta<t,$
\begin{equation}\label{functionlookingfor}
\begin{split}
 \int_\delta^{t} \!\!\!\|AT(t-s)\|_{L(X)} \|g(t)\!-\!g(s)\|_X ds &  \!\leqslant\!\! M_1  \int_\delta^{t}\!\! (t-s)^{-1} \!(w ( C\delta^{-\beta}(t-s)^{\beta} )+ w ( (t-s)^\beta ))ds \\\\
& =M_1 \int_{0}^{\delta^{-1} C^{1/\beta}(t-\delta)} \dfrac{ w(u^\beta)}{u}du \!+\!  M_1 \! \int_{0}^{t} \!\!\dfrac{w(u^\beta)}{u}du  < \infty,
\end{split}
\end{equation}
and
\begin{equation*}
 \int_0^\delta \|AT(t-s)\|_{L(X)} \|g(t)-g(s)\|_X ds   \leqslant 2M_1  \int_0^\delta (t-s)^{-1} B ds=2M_1\ln\frac{t}{(t-\delta)}.
 \end{equation*}
Moreover, we have that
\[
\begin{array}{rl}
& \|h^{-1} \displaystyle  \int_{0}^{t} ( T(h)-I + hA )T(t-s)(g(s)-g(t))ds \|_X \\\\ 
& \leqslant \displaystyle  \int_{0}^{t} \| h^{-1}   \displaystyle  \int_{0}^{h} ( I-T(\sigma))d\sigma A T(t-s) (g(s)-g(t))\|_X ds .
\end{array}
\]
Since, for $ \delta\leqslant s\leqslant t$,
$$
 \| h^{-1}   \int_{0}^{h} ( I-T(\sigma))d\sigma A T(t-s) (g(s)-g(t))\|_X\leqslant
 2M_0M_1 (t-s)^{-1} \!(w ( C\delta^{-\beta}(t-s)^{\beta} )+ w ( (t-s)^\beta )),  
 $$
 for $0\leqslant s\leqslant \delta$,
 $$
 \| h^{-1}   \int_{0}^{h} ( I-T(\sigma))d\sigma A T(t-s) (g(s)-g(t))\|_X  \leqslant 2M_1  (t-s)^{-1} B,
 $$
 and for each $s\in (0,t]$
$$
h^{-1}   \displaystyle  \int_{0}^{h} ( I-T(\sigma))d\sigma A T(t-s) (g(s)-g(t)) \xrightarrow{h\rightarrow 0^+} 0,
$$
it follows from the Dominated Convergence Theorem that
$$
\|h^{-1} \displaystyle  \int_{0}^{t} ( T(h)-I + hA )T(t-s)(g(s)-g(t))ds \|_X \xrightarrow{h\rightarrow 0^+} 0 ,
$$
uniformly for $t\in [\delta, T_1]$, for any $\delta>0$.

To see that the second integral in (\ref{AG(t)}) converges, note that
\begin{equation*}
\begin{split}
h^{-1} \displaystyle  \int_{0}^{t} ( T(h)-I )T(t-s) &g(t)ds  =h^{-1} \displaystyle  \int_{0}^{h} ( T(h)-I )T(t-s)g(t)ds \\
&  +h^{-1} \displaystyle  \int_{h}^{t} ( T(h)-I )T(t-s)g(t)ds \\ & = h^{-1}\displaystyle  \int_{0}^{h} T(t+h-s) g(t)ds - h^{-1}\displaystyle  \int_{t-h}^{t}T(t-s)g(t)ds \\
& \xrightarrow{h \rightarrow 0^+}( T(t)-I)g(t),
\end{split}
\end{equation*}
with the convergence being uniform for $t\in [0,T_1]$. Therefore
$$
h^{-1} ( T(h)-I )G(t) \xrightarrow{h \rightarrow 0^+} - \displaystyle \int_{0}^{t}AT(t-s) [ g(s)-g(t)]ds + ( T(t)-I)g(t)
$$
uniformly on $[\delta,T_1]$. To prove that $AG(t)$ is continuous on $(0,T_1]$, we write 
\begin{equation}
\begin{array}{rl}
G(t) =  \displaystyle \int_{0}^{t}T(t-s)g(s) ds 
& = \displaystyle \int_{0}^{t}T(t-s)(g(s) - g(t) )ds + \displaystyle \int_{0}^{t}T(t-s) g(t)ds \\\\
& = G_1(t) + G_2(t).
\end{array}
\end{equation}
We have seen that $G_2(t) \in D(A)$ and that $AG_2(t) = (T(t)-I)g(t)$; since $f$ and $u$ are continuous, it is easy to see that $AG_2(t)$ is continuous on $[0,T_1]$.

To conclude the proof we have to show that $AG_1(t)$ is continuous on $(0,T_1]$. For $0<\delta<t$ we have
\begin{equation}
\label{AG_1}
AG_1(t)=\int_{0}^{\delta}AT(t-s)(g(s) - g(t) )ds + \displaystyle \int_{\delta}^{t}AT(t-s)(g(s) - g(t) )ds
\end{equation}
and, for $\beta>0$,
\begin{equation}
\begin{array}{rl}
&\| \displaystyle \int_{0}^{\delta} AT(t+h-s)(g(s)-g(t+h))ds - \displaystyle \int_{0}^{\delta} AT(t-s)(g(s)-g(t))ds \| \\\\
& = \| \displaystyle \int_{0}^{\delta}AT(t-s)g(t)ds - \displaystyle \int_{0}^{\delta} AT(t+h-s)g(t+h)ds + \displaystyle \int_{0}^{\delta}A(T(t+h-s)-T(t-s))g(s)ds\| \\\\
& \leqslant \| (T(t-\delta) -I)g
(t) - (T(t+h-\delta) -I)g(t+h)\| + B \displaystyle  \int_{0}^{\delta}\| A( T(t+h-s)-T(t-s))\|ds \\\\
& = \| (T(t-\delta) -I)g
(t) - (T(t+h-\delta) -I)g(t+h)\| + B \displaystyle  \int_{0}^{\delta}\|  (T(h)-I)A^{-\beta}A^{1+\beta}T(t-s))\|ds \\\\
& \leqslant \| (T(t-\delta) -I)g
(t) - (T(t+h-\delta) -I)g(t+h)\| + B\frac{M_{1-\beta}M_{1+\beta}}{\beta}h^{\beta} \displaystyle  \int_{0}^{\delta}  (t-s)^{-1-\beta}ds \\\\
& \xrightarrow{h\rightarrow 0^+} 0,
\end{array}
\end{equation}
uniformly for $ t\in [0,T_1]$ and, for $h<\delta$,

\begin{equation}
\begin{split}
\|  \int_{\delta}^{t+h} &AT(t+h-s)(g(s)-g(t+h))ds -  \int_{\delta}^{t} AT(t-s)(g(s)-g(t))ds \| \\
&  =  \|  \int_{\delta - h}^{t}       AT(t-s) (g(s+h) - g(t+h))ds -    \int_{\delta}^{t} AT(t-s)(g(s)-g(t))ds \|  \\
& \leqslant  \|   \int_{\delta}^{t}\!\!AT(t-s)(g(s+h) \!-\! g(t+h)\!-\!g(s)\!+\!g(t))ds \| \\
&+\int_{\delta - h}^{\delta} \!\!\!\! AT(t-s) (g(s+h)\!-\! g(t+h))ds\| .
\end{split}
\end{equation}
Now, since
\begin{equation}
\begin{split}
\displaystyle  \int_{\delta}^{t} \| AT(t-s) & (g(s+h) - g(t+h)-g(s)+g(t))\|ds \\
            & \leqslant 2\displaystyle  \int_{\delta}^{t}(t-s)^{-1}(w(C\delta^{-\beta}(t-s)^\beta)+ w((t-s)^\beta)ds < \infty,
\end{split}
\end{equation}
and, for each $s\in (0,t]$, we have 
$$
\| AT(t-s)(g(s+h) - g(t+h)-g(s)+g(t))\| \xrightarrow{h\rightarrow 0^+} 0,
$$
and it follows by Lebesgue's Dominated Convergence Theorem  that 
$$
\displaystyle  \int_{\delta}^{t} \| AT(t-s)(g(s+h) - g(t+h)-g(s)+g(t))\|ds \xrightarrow{h\rightarrow 0^+} 0,
$$
and since 
\begin{equation}
\begin{split}
\|\int_{\delta - h}^{\delta} AT(t-s) &(g(s+h)- g(t+h))ds\|\\
&\leqslant M_1  \int_{_{\delta - h}}^{_{\delta}}(t-s)^{-1}\! (w(C\delta^{-\beta}(t-s)^{\beta})\!+\! w((t-s)^\beta))ds  \\
            &\leqslant M_1    \int_{\delta^{-1} C^{1/\beta}(t-\delta)}^{\delta^{-1} C^{1/\beta}(t-\delta+h)}\dfrac{w(u^\beta)}{u} du  \xrightarrow{h\rightarrow 0^+} 0,
\end{split}
\end{equation}

we conclude that the second integral on the right hand side of (\ref{AG_1}) is a continuous function of $t$. Thus, $AG_1(t)$ is continuous and therefore $AG(t)$ is a continuous function on $(0,T_1]$.

\end{proof}

In \cite[Theorem 4.3.2]{pazy}, it was shown that the above result holds if $ A$ is the infinitesimal generator of an analytic semigroup, for functions $f$ that only depend on time and $f \in L^1(0,T;X)$ and such that, for  each $t\in (0,T)$ there is a $\delta_t$ and a continuous real value function $W_t:[0,\infty) \rightarrow [0,\infty)$  such that 
$$
\|f(t)-f(s)\| \leqslant W_t(|t-s|) 
$$
and
$$
\displaystyle \int_{0}^{\delta_t} \dfrac{W_t(\tau)}{\tau}d\tau < \infty.
$$
We extend this result to the case when $f$ also depends on state.

As an immediate consequence of the previous results we obtain that

 \begin{corol}
If $X, A$ and $\{T(t):t \geqslant 0\}$ are as in the Theorem \ref{mildsolution} and $f: X^{\alpha} \rightarrow X, \, 0\leqslant \alpha < 1,$ is a continuous map, the initial value problem
\begin{equation}\label{autonmous}
\left\{ \begin{array}{lcc}
\dfrac{ du(t)}{dt}+  Au(t)= f(u)  \,\,\,\,\,\,\,\,\,\, \thinspace t>0, \,\, \\
\\ u(0)=w_0 \in X^\alpha,
\end{array}
\right.
\end{equation}
possesses a mild solution for each $w_0 \in X^\alpha$. If $f$ is such that for each $u\in C([0,T],X^{\alpha})$
\begin{equation}\label{hipsobrefautonoma}
\|f(u(t))-f(u(s))\|_X \leqslant w(\|u(t)-u(s)\|_{\alpha}),
\end{equation}
where $w:[0,\infty) \rightarrow [0,\infty)$ is a continuous increasing function, $w(0)=0$ and
\begin{equation}\label{hipsobrewcasoautonomo}
\displaystyle \int_{0}^{t}u^{-1}w(u^\beta)du < \infty,
\end{equation}
  then \eqref{autonmous} has a strong solution for each $w_0\in X^\alpha$.
  \end{corol}

\begin{obs}
The condition imposed on the function $w$ is strictly weaker than holder continuity. In fact, a possible choice of $w$ is $w(x)=( \ln(|x|^{-1}))^{-p}$, $x>0$, $p>1$. In that case $w:(0,e^{-1})\to \R$ is strictly increasing, 
$$
 \int_{0}^{e^{-1}} \frac{1 }{x(\ln(x^{-1})^{p}}dx =\int _1^\infty u^{-p} du <\infty
 $$
 and, for $r>0$,
$$
\lim_{x \rightarrow 0} \frac{x^r}{w(x)}=\lim_{x \rightarrow 0 } \dfrac{x^r}{( \ln(|x|^{-1}))^{-p}}  = \lim_{x \rightarrow 0 }( x^{r/p} (-\ln|x|) )^{p}=0.
$$

\end{obs}

 \begin{obs}\label{mildandstrongsolutionexistence}     

 Consider the initial boundary value problem 
 \begin{equation}\label{problematimereparametrizado}
\left\{ \begin{array}{lcc}
\dfrac{\partial v}{\partial t } = \dfrac{\partial v^2}{\partial x^2}+  \frac{1}{a(l(v))}\left(\lambda f(v) + h(\sigma+\int_\sigma^t a(l(v(\theta)))d\theta)\right), \,\,\,\,\,\,\,\,\,\, \thinspace t  >\sigma, \,\, x \in \Omega,  \\
\\ v(t,\sigma,0)=v(t,\sigma,1)=0,\\
\\ v(\sigma,\sigma,x)=w_0(x).
\end{array}
\right.
\end{equation}

If $ f,h$ and $a\circ l$ are continuous functions we obtain the existence of mild solutions for \eqref{problematimereparametrizado} by Theorem \ref{mildsolution}. To obtain a regular solution we use  Theorem \ref{strongsolution}: assume that $a,h \in C(\R^+), \, a(s)\geqslant m > 0$, $l:H^1_0(\Omega) \rightarrow \R^+$ is a continuous functional and $f:H^1_0(\Omega) \rightarrow L^2(\Omega),$ a continuous function with the property that, for each $w_0\in X^\alpha$   there is a neighborhood $V_{w_0}$ of $w_0$ in $X^\alpha$ such that
$$
\| f(u_1)-f(u_2) \| \leqslant w(\| u_1-u_2\|_\alpha) , \,\,
\left\| a(l(u_1))-a(l(u_2)) \right\| \leqslant w(\| u_1-u_2\|_\alpha), \ \hbox{ for all } u_1,u_2\in V_{w_0},
$$
and 
$\|h(t)-h(s)\| \leqslant w(|t-s|^\beta),
$ where $w:[0,\infty) \rightarrow [0,\infty)$ is increasing, continuous, $w(0)=0$ and

$$
 \displaystyle \int_{0}^{t}u^{-1}w(u^{\beta})du < \infty ,\hbox{ for some }\beta \in(0,1- \alpha).
$$ 
Therefore, problem \eqref{problematimereparametrizado} has a strong solution which turns out to be (by simple bootstrapping arguments) classical.\\
Furthermore, as discussed in Section $2,$ making the reparameterization  $t = \sigma+\int_s^\tau a(l(v(\theta)))^{-1}d\theta$ we have that $w(\tau,x)=v(t,x)$ is solution of $\eqref{problema}.$ That is, we have proved the existence and regularity for a solution of the problem \eqref{problema} thanks to the results obtained for \eqref{problematimereparametrizado}.
 \end{obs}

\section{Comparison results, bounds and global solutions.}\label{comparison_results}

In this section, we restrict our attention to spaces with a partial ordering. Typically these spaces are linear subspaces of the space of functions defined in a domain $\Omega\subset \R^N$ and taking value in $\R$ which are locally integrable ($L^1_{\rm loc}(\Omega)$)  like $X=L^2(\Omega)$ or $X^{1/2}= 
H^{1}_{0}(\Omega)$. In this situation the partial ordering is defined 
$$
u \leqslant v  \Leftrightarrow u(x) \leqslant v(x) ~ a.e. ~ for ~ x ~ in ~ \Omega.
$$
We aim to prove comparison results for our prototype model \eqref{problemau} and similar models without uniqueness. Our approach is abstract, based on \cite{acrb-2000} (see also \cite{carvalho2012attractors}) and extended to include semilinear problems without uniqueness.

As before we assume that $-A$ has compact resolvent, is the infinitesimal generator of an  analytic semigroup $\{e^{-At}:t\geqslant 0\}$ and that $\alpha\in (0,1)$. 

We also
assume that $X$ has a partial ordering $\leqslant$  and that the positive cone in $X^\alpha$, that is,
$$
C_\alpha=\{x\in X^\alpha: 0\leqslant x\}, \ \alpha \geqslant 0
$$
is closed. We also require that $-A$ has positive resolvent, that is,  $(-A)^{-1} C_\alpha \subset C_\alpha$.

We prove comparison results for the initial value problems of the form
\begin{equation}\label{g+-}
\left\{ \begin{array}{lcc}
\dot{u}= Au + g^{\pm}(u) , \,\,\,\,\,\,\,\,\,\, \thinspace t>0,
\\ u(0)=w_0\in X^\alpha
\end{array}
\right.
\end{equation}
where $g^{\pm} : X^\alpha  \rightarrow X$ are continuous maps. 

Assuming that $g^{-}\leqslant g^{+}$  and that, for some choice of $k\in \R$, $g^{+}(u) + k u$ is an increasing function, we will prove that for each solution $u^{\pm}(t,w_0)$ exist $u^{\mp}(t,w_0)$ such that $u^{-}(t,w_0)\leqslant u^{+}(t,w_0)$.

 By using the comparison results in this section, we will be able to ensure that the solutions of \eqref{problematimereparametrizado} are globally defined, therefore those of \eqref{problema} are also globally defined.

\begin{theo}\label{comparison}
 Let $g^{\pm}:X^\alpha \rightarrow X$ be continuous maps such that  $ g^- \leqslant g^+$ and $g^{+}(u) + k u$ is an increasing function. Then

\begin{itemize}
    \item [(i)] For each solution $u(t,w_0,g^+)$ exist $u(t,w_0,g^-)$ such that $u(t,w_0,g^-)\leqslant u(t,w_0,g^+)$ as long as both are defined.

\item[(ii)] Given a solution $u(t,u^+\!,g^+\!)$ and $u^-\in X^\alpha$ with $u^-\! \leqslant u^+\!$, there exists a solution $u(t,u^-\!,g^+\!)$ such that $u(t,u^-\!,g^+\!)\leqslant u(t,u^+\!,g^+\!)$ as long as both are defined. 

\end{itemize}
An analogous result holds if $g^{-}+kI$ is an increasing function.
\end{theo}

\begin{proof}

\begin{itemize}
    \item[(i)] We write the initial value problem \eqref{g+-} (with forcing term $g^+$) as
\begin{equation}\label{traslacionidentidad}
\left\{ \begin{array}{lcc}
\dfrac{ du(t)}{dt}= -(A+kI)u + g^{+}(u) + ku , \,\,\,\,\,\,\,\,\,\, \thinspace t>0, \,\, \\
\\ u(0)=w_0 \in X^\alpha.
\end{array}
\right.
\end{equation}

Note that $k\in \R$ may be chosen such that $(g^++kI)(\cdot)$ is increasing and $A+kI$ is the generator of an exponentially decaying semigroup. Since $g^{+}$ is continuous, follows that $g^{+}(u) + k u$ is an continuous function, it follows from Theorem \ref{mildsolution} that the initial value problem \eqref{traslacionidentidad} has a mild solution $u^{+}(\cdot)$.

Let $T(t)$ be the semigroup generated by $-A-kI $, then \eqref{asg_est} and \eqref{asg_holder} hold.

$$
\| A^{\alpha} T(t) \| \leqslant M_\alpha t^{-\alpha}e^{-\omega t}\,\,\, for \,\,t>0.
$$

Fix a mild solution  $u^+(\cdot) \in C([0,T_1],X^\alpha)$ of  \eqref{traslacionidentidad} and choose $T_1\in (0,t_1]$ and $\rho>0$ such that
$$
\| T(t)w_0-w_0 \|_{\alpha} \leqslant \rho/2 \,\,\, for \,\, 0 \leqslant t \leqslant T_1,
$$ 
$$
\| g^{-}(v) + kv\|_{X}\leqslant N_{T_1}  \,\,\,\, for \,\,all\,\, v \in  B_\rho^{X^\alpha}(w_0)=\{v_0: \, \| v_0-w_0\|_{X^\alpha}  \leqslant \rho \}, 
$$
$$
N_{T_1}M_\alpha \displaystyle \int_0^t u^{-\alpha}e^{-\omega u}du \leqslant \rho/2 \,\,\, for \,\, 0 \leqslant t \leqslant T_1 \hbox{ and } M_\alpha N_{T_1} \rho/2 \geq T_1.
$$

Set $Y= C([0,T_1];X^\alpha)$ and 
$Y_{0}=\{ u^{-}: u^{-} \in Y, u(0)=w_0, u^{-}(t) \in B_\rho^{X^\alpha}(w_0), u^-(t) \leqslant u^+(t),~ 0 \leqslant t \leqslant T_{1}\}.
$\
We define a mapping $F_{w_0}:Y_0 \rightarrow Y$ by 
$$
(F_{w_0}u^{-})(t)= T(t)w_0 \; + \; \int_{0}^{t}  T(t-s) g^{-}(u^{-}(s)) + ku^{-}(s) )ds.
$$
Note that, for all $u^{-}\in Y_0$,
\[
\begin{array}{rl}
\| (F_{w_0}u^{-})(t)-w_0 \|_{\alpha}& \leqslant  
\| T(t)w_0-w_0  \|_{\alpha} \; + \displaystyle \int_{0}^{t} \| T(t-s) (g^{-}(u^{-}(s))+ ku^{-}(s) )\|_{\alpha} ds.\\\\
& \leqslant  \rho,
\end{array}
\]
and since  $g^{+}(u) + k u$ is an increasing function we have 
\[
\begin{array}{rl}
(F_{w_0}u^{-})(t)-u^{+}(t)= & \!\! \!\!\!\displaystyle  \int_{0}^{t}\!\!  T(t-s) (g^{-}(u^{-}(s))\!+\! ku^{-}(s))ds
-  \!\!\displaystyle  \int_{0}^{t} \!\! T(t-s) (g^{+}(u^{+}(s))\!+\!k u^{+}(s))ds \\\\
\leqslant &  \displaystyle  \int_{0}^{t}  T(t-s)( g^{+}(u^{-}(s))+ ku^{-}(s)- g^{+}(u^{+}(s)) -k u^{+}(s))ds\\\\
\leqslant & \displaystyle  \int_{0}^{t}  T(t-s)( g^{+}(u^{+}(s))+k u^{+}(s)- g^{+}(u^{+}(s))-k u^{+}(s))ds =0.
\end{array}
\]
It follows that $F_{w_0}$ maps $Y_0$ into $Y_0$. Also, the continuity of $g^-(u^-)+ku^-$ implies that $F_{w_0}$ is continuous. Moreover, as before, we conclude that  $\overline{F_{w_0}(Y_0)}$ is compact into $Y_0$.

Schauder's fixed point Theorem implies that $F$ has a fixed point in $Y_0$ and it is a mild solution of \eqref{g+-} (with forcing term $g^-$) on $C([0,T_1],H^{1}_{0}(\Omega))$. This procedure successively repeated at the end time of the interval proves the result.

\item[(ii)] Fix $u^+\in X^\alpha$ and write equation \eqref{g+-} as
			$$
			u_t=-(A+\beta I)u+ g^+(u)+\beta u.
			$$			
Given a mild solution $u(\cdot,u^+,g^+)\in C([0,T_1],X^\alpha)$ we recursively build the following sequence 
			\begin{equation}
\begin{split}
					u_1(t)\!=& F_{u^-}(u(t,u^+\!,g^+\!))(t)\!=\!T(t)u^-\!+\! \displaystyle \int_0^t T(t-s)[g^+(u(s,u^+\!,g^+))+\beta u(s,u^+\!,g^+)]ds; \\\\
					u_2(t)= &F_{u^-}(u_1)(t)=T(t)u^-+ \displaystyle \int_0^t T(t-s)[g^+(u_1(s))+\beta u_1(s)]ds; \\
					\vdots & \\
					u_{n+1}(t)= & F_{u^-}(u_n)(t)=T(t)u^-+ \displaystyle \int_0^t T(t-s)[g^+(u_n(s))+ \beta u_n(s)]ds;
\end{split}
			\end{equation}			
where $T(t)$ is the semigroup generated by $-A-\beta I$. Since $T(t)u^+ \geqslant T(t)u^-$ we have that $u_1(t) \leqslant u(t,u^+,g^+)$, due to $g^+(\cdot) + \beta I $ is an increasing function it follows that $u_2(t)=F_{u^-}(u_1(t))\leqslant F_{u^-}(u(t,u^+,g^+))= u_1(t) \leqslant u(t,u^+,g^+)$ and repeating the same argument we conclude that $u_n(t) \leqslant u(t,u^+,g^+)$.\\
Moreover, as before, we conclude that $\overline{\{u_n: n \in \N\}}$ is compact in $Y_0$. Then there exists a subsequence that we denote $\{u_n\}_{n \in \N} $ (relabeled the same) that converges to $u_*$ and since 
			$$
			\lim_{n \rightarrow \infty} (F_{u^-}u_n)(t)= \lim_{n \rightarrow \infty} u_{n+1}(t)
			$$
			then $(F_{u^-}u_*)(t)\!=\!u_*(t)$ and since $(F_{u^-}u_n)(t)\! \leqslant\! u(t,u^+,g^+)$ it follows that $u_*(t)=u(t,u^-,g^+)$ $ \leqslant u(t,u^+,g^+) $. 		

 \end{itemize}
\end{proof}

\bigskip

In what follows we consider $g$ instead of $g^{+} $ or $g^{-}$ in \eqref{g+-} in order to be able to use it when either $(g^{-}+kI)(\cdot)$  is increasing or when $(g^{+}+kI)(\cdot)$ is increasing.

\begin{theo}
Assume that for every $r>0$ there exists a constant $\beta=\beta(r)>0$ such that  $g(\cdot)+ \beta I $ is positive in the positive elements in $B_r^{X^\alpha}(0)$ and that $g$ is a continuous function. Then, if $0 \leqslant w_0 \in X^\alpha$, the solution of \eqref{g+-} $u(t,w_0,g)$ is positive as long as it exists.
\end{theo}

\begin{proof}
For $w_0 \in X^\alpha$, we write the equation as
$$
u_t=-(A+\beta I)u+ g(u)+\beta u
$$
and recursively define the sequence $\{u_i: \;i \in \N  \}$
\begin{equation}
\begin{array}{rl}
u_1(t)= & F(w_0)(t)=e^{-(A+\beta I)t}w_0+ \displaystyle \int_0^t e^{-(A+\beta )(t-s)}[g(w_0)+\beta w_0]ds \\\\
u_2(t)= &F(u_1)(t)=e^{-(A+\beta I)t}w_0+ \displaystyle \int_0^t e^{-(A+\beta )(t-s)}[g(u_1(s))+\beta u_1(s)]ds \\
\vdots & \\
u_{i+1}(t)= & F(u_i)(t)=e^{-(A+\beta I)t}w_0+ \displaystyle \int_0^t e^{-(A+\beta )(t-s)}[g(u_i(s))+\beta u_i(s)]ds
\end{array}
\end{equation}
Note that $u_i \geqslant 0$, $i \in \N$ since $w_0 \geqslant 0$ and $g(\cdot)+ \beta I $ is positive in the positive elements in $B_r^{X^\alpha}(0)$.\\
Let $T_1 >0$ be such that 
$$
\| T(t)w_0-w_0 \|_{X^\alpha} \leqslant \rho/2 \,\,\, for \,\, 0 \leqslant t \leqslant T_1,
$$
$$
\| g(u_i) + ku_i\|_{X}  \leqslant N_{T_1}  \,\,\,\, for \,\,all\,\, i \in \N,
$$
$$
N_{T_1}M_\alpha \displaystyle \int_0^t u^{-\alpha}e^{-\omega u}du \leqslant \rho/2 \,\,\, for \,\, 0 \leqslant t \leqslant T_1 \hbox{ and } T_{1}\leqslant M_\alpha N_{T_1} \rho/2 .
$$

Set $Y= C([0,T_1],X^\alpha)$ and $Y_{0}=\{v \in C([0,T_1],X^\alpha): v(0)=w_0 \hbox{ and }  v(t)\geqslant 0 \}$.

Clearly $u_{i+1}=(Fu_i)(t)\geq 0$ for all $t\in [0,T_1]$. As before, $\overline{\{u_i:\; i \in \N \}}$ is compact in $Y_0$ and there exists a subsequence that we denote $\{u_i\}_{i \in \N} $ (relabeled the same) that converges to some function $u\in Y_0$  uniformly in $[0,T_1]$ and since 
$$
\lim_{i \rightarrow \infty} F(u_i)(t)= \lim_{i \rightarrow \infty} u_{i+1}(t),
$$
it follows that $ Fu=u$ and that $u(t) \geqslant 0 $.

With a continuation argument, we can show that $ u(t) \geqslant 0$ as long as the solution is defined.
\end{proof}

In a similar way we can obtain the following result.
 
\begin{theo}\label{corolario12.6}
\begin{itemize}
\item[(i)] Assume that $g(C_\alpha)\subset C_0$. Then $w_0 \geqslant 0$ implies that any solution $u(\cdot)$  of  \eqref{g+-} (with $g$ instead of $g^{\pm})$ must satisfy $u(t) \geqslant 0$  as long as the solution exists.

\item[(ii)] Let $u^+ \geqslant u^-$ be initial conditions in $X^\alpha$, $g^+, \; g^-$ continuous functions such that $g^+ \geqslant g^-$ and either $g^+(\cdot)+ \beta I $ or $g^-(\cdot)+ \beta I $ is an increasing function. If $u(t,u^+,g^+),\; u(t,u^-,g^-)$ are solutions of the respective initial value problem in \eqref{g+-}, then there are solutions $u^*(t,u^+,g^+)$, $u_*(t,u^-,g^-)$ such that $ u(t,u^+,g^+)\geqslant u_*(t,u^-,g^-) $ and $u^*(t,u^+,g^+) \geqslant u(t,u^-,g^-) $.

 \item[(iii)] If $g^{-}, g^{+}$ are such that for every $r>0$ there exist a constant $\beta= \beta(r)>0$ and an increasing function $h(\cdot)$ such that 
 $$
 g^{-}(\cdot) + \beta I \leqslant h(\cdot) \leqslant  g^{+}(\cdot) + \beta I 
 $$
 in $B_r^{X^\alpha}(0)$ ; $u^-,u^+ \in X^\alpha$, $u^- \leqslant u^+,$ and $u(t,u^+,g^+),\; u(t,u^-,g^-)$ solutions of the respective initial value problem in \eqref{g+-}, then there are solutions $u^*(t,u^+,g^+)$, $u_*(t,u^-,g^-)$ such that $u^*(t,u^+,g^+)\geqslant  u(t,u^-,g^-)$ and  $u(t,u^+,g^+) \geqslant u_*(t,u^-,g^-) $.
 \end{itemize}  
\end{theo}

\begin{proof}
\begin{itemize}
\item[(i)]Since $w_0 \geqslant 0$ and $g(u(t)) \geqslant 0$ we have  $T(t)w_0$ and $T(t-s)g(u(s)) \geqslant 0$ then $u(t) \geqslant 0$ for all  $t \in [0,T_1]$. A continuation argument shows that this inequality is true as long as the solution exists.

\item[(ii)] Let $h+\beta I$ be increasing with either $h=g^+$ or $h=g^-$. Given $u(\cdot,u^-,g^-)$, from Theorem \ref{comparison}, there is a $u(\cdot,u^-,h)$ with $u(t, u^-,h)\geqslant u(t, u^-,g^-)$ and using that $h+\beta I$ is increasing there is a $u(\cdot,u^+,h)$ such that $u(t, u^+,h)\geqslant u(t, u^-,h)\geqslant u(t, u^-,g^-)$ and finally there is a $u^*(\cdot,u^+,g^+)$ such that 
$u^*(t,u^+,g^+)\geqslant u(t, u^+,h)\geqslant u(t, u^-,h)\geqslant u(t, u^-,g^-)$. In a similar way, using that $h+\beta I$  is increasing, from Theorem \ref{comparison}, we obtain  $u(\cdot, u^+,h)$, $u(\cdot, u^-,h)$ and $u_*(\cdot, u^-,g^-)$ such that
$u(t,u^+,g^+)\geqslant u(t, u^+,h)\geqslant u(t, u^-,h)\geqslant u_*(t, u^-,g^-)$.

   \item[(iii)] After adding and subtracting the term $\beta u $ in the equation \eqref{g+-} the result follows from (ii).

\end{itemize}
\end{proof}

The next Corollary plays an important role in the study of the asymptotic behavior since it is the key result to ensure that the solutions of \eqref{problematimereparametrizado} are globally defined as well as to ensure that they are bounded.

\begin{corol} \label{corolariocomparacion}   Assume that $a\in C(\R^+, [m,M])$ for some choice of $m,M\in (0,\infty)$ and that $f:\R \to \R$ is a continuous function. Assume also that $f$ is such that for every $r>0$ there exists a constant $\gamma= \gamma(r) > 0 $ such that $\gamma I+\lambda f(\cdot)$ is increasing in $(-r,r)$. Let $u_0,u_1, u_2 \in H^{1}_{0}(\Omega), u_0 \leqslant u_1 \leqslant u_2$.

If $g^-:=\frac{\lambda f(\cdot) -K}{M} $ and $g^+:=\frac{\lambda f(\cdot) +K}{m} $ where $K:=\sup_{s}|h(s)|$, for any solution $v(t,\sigma,u_1,g)$ of \eqref{problematimereparametrizado} there are solutions $w(t,\sigma,u_0,g^-)$, $w(t,\sigma,u_1,g^-)$, $z(t,\sigma,u_1,g^+)$, $z(t,\sigma,u_2,g^+)$ such that
\begin{equation*}
\begin{split}
&w(t,\sigma,u_0,g^-)\leqslant w(t,\sigma,u_1,g^-) \leqslant v(t,\sigma,u_1,g)\\
&v(t,\sigma,u_1,g)\leqslant z(t,\sigma,u_1,g^+)\leqslant z(t,\sigma,u_2,g^+).
\end{split}
\end{equation*}
\end{corol}

\begin{proof}
Consider the auxiliary initial boundary value problems

\begin{equation}\label{f-}
\left\{ \begin{array}{lcc}
w_t = w_{xx} + \frac{\lambda f(w) - K}{M} , \,\,\,\,\,\,\,\,\,\, \thinspace t >\sigma , \,\, x \in \Omega,  \\
\\ w(t,\sigma,0)=w(t,\sigma,1)=0,\\
\\ w(\sigma,\sigma,x)=w_0(x),
\end{array}
\right.
\end{equation}
and
\begin{equation}\label{f+}
\left\{ \begin{array}{lcc}
z_t = z_{xx} + \frac{\lambda f(z) +K}{m} , \,\,\,\,\,\,\,\,\,\, \thinspace t >\sigma , \,\, x \in \Omega,  \\
\\ z(t,\sigma,0)=z(t,\sigma,1)=0,\\
\\ z(\sigma,\sigma,x)=w_0(x).
\end{array}
\right.
\end{equation}

Observe that, since $H^1_0(0,1)\subset L^\infty(0,1)$ with continuous embedding, given $r>0$ it is easy to see that there exist $\gamma(r)>0$ such that for $\nu \in \R^{+}$ and $\|u\|_{H^1_0(0,1)} \leqslant r$

\begin{equation}\label{comparing3functions}
\gamma u + \frac{\lambda f(u) - K}{M} \leqslant \gamma u + \frac{\lambda f(u) + h(s)}{a(\nu)} \leqslant \gamma u + \frac{\lambda f(u) +K}{m} , \ u\in B_r^{H^1_0(0,1)}(0)
\end{equation}
with 
$\gamma u + \frac{\lambda f(u) - K}{M} $ and $ \gamma u + \frac{\lambda f(u) +K}{m} $ being increasing functions in the variable $u\in B_r^{H^1_0(0,1)}(0).$\\

The result now follows by successive applications of Theorem \ref{corolario12.6} with
\begin{equation}\label{gpmeg}
\begin{split}
&  g^{-}(t,u)  =\dfrac{\lambda f(u) - K}{M},\\
& g^{+}(t,u)= \dfrac{\lambda f(u) +K}{m} ,\\
& g(t,u)=\dfrac{\lambda f(u) + h(\sigma+\int_\sigma^t a(l(u(\theta)))d\theta)}{a(l(u))}.
\end{split}
\end{equation}
\end{proof}

To prove that all solutions of \eqref{problemau} are globally defined, we need to impose a structural condition (\textbf{S})  on the non-linearity $f$.

\bigskip

To prove the existence of a pullback attractor we assume in addition to (\textbf{S}) a dissipativity condition (\textbf{D}) on $f$. 

\bigskip

For the rest of this section $Au=u_{xx}$ with Dirichlet boundary conditions in the interval $(0,1)$ and we will use the notation of Corollary \ref{corolariocomparacion} for $g,g^+$ and $g^-$. We apply the comparison results of Theorem \ref{corolario12.6}  to $g^-$ and $g$ or $g$ and $g^+$ with $h$ being, respectively, $g^-$ or $g^+$.

\bigskip

We start recalling the following lemma from \cite{acrb-2000}.

\begin{lema}\label{Lema12.7}
Let $\lambda_1$ be the first eigenvalue of $A+\nu C_0I$, which we assume to be positive. If $\alpha_0 >0$ then there is a constant $M=M(A,\lambda_1,\alpha_0)$ such that for every $\alpha_0 \geqslant \alpha \geqslant \beta \geqslant -\alpha_0$
$$
\|e^{-(A+\nu C_0I)t}u_0\|_\alpha \leqslant Me^{-\lambda_1 t} t^{-(\alpha-\beta)}\|u_0\|_\beta   \;\;\;\;\;t>0,\;\;\;u_0 \in X^\beta.
$$
Furthermore, given $q> 1$ there exist a $\delta_0\in (0,1)$ and a constant $\tilde{M}=\tilde{M}(A,\lambda_1)$ such that 
$$
\|e^{-(A+\nu C_0I)t}u_0\|_{L^\infty(\Omega)} \leqslant \tilde{M}t^{-\delta_0}e^{-\lambda_1t} \|u_0\|_{L^{q}(\Omega)}\;\;\;\;\;t>0,\;\;\;u_0 \in L^{q}(\Omega)
$$
\end{lema}

Now, we are going to prove global existence of solutions and obtain bounds on the global attractors assuming conditions $\mathbf{(S)}$ and $\mathbf{(D)}$ as needed. A very important preliminary result is a pointwise bound on the global attractor as given by the following result. 

\begin{prop}\label{proposicion12.8}
Let $w_0\in H^1_0(0,1)$, $\sigma\in \R$ and let $u(\cdot,\sigma, w_0)$ be a solution of \eqref{problemau}.
\begin{itemize}
\item[(i)] If (\textbf{S}) holds, then $u(\cdot,\sigma,w_0)$ is globally defined.
\item[(ii)]If (\textbf{D}) holds, there are constants $K_\infty$ and $K_\alpha$,   $\alpha \in (0,1)$ such that 
\begin{equation}\label{k_infty}
\limsup_{\sigma\longrightarrow- \infty }\|u(t+\sigma,\sigma,w_0)\|_{L^{\infty}(\Omega)} \leqslant K_\infty,
\end{equation}

\begin{equation}\label{k_r}
\limsup_{\sigma\longrightarrow -\infty }\|u(t+\sigma,\sigma,w_0)\|_{\alpha } \leqslant K_\alpha,
\end{equation}
and these limits are uniform for $w_0$ in bounded subsets of  $H^{1}_{0}(\Omega)$ and for $\sigma\in \R$.
\item[(iii)] In addition, if (\textbf{D}) holds, $w_0 \in H^{1}_{0}(\Omega)$ and $\phi$ is the solution of 
\begin{equation}\label{A+C_0}
\left\{ \begin{array}{lcc}
(A+\nu C_0)\phi=C_1 , \,\,\,\,\,\,\,\,\,\, \thinspace in\;\; \Omega,  \\
\\ \phi=0, \;\;\;\;\ in\;\; \partial \Omega,\\
\end{array}
\right.
\end{equation}
then $0 \leqslant \phi \in L^{\infty}(\Omega), \; \limsup_{t\rightarrow \infty} |u(t+\sigma,\sigma,w_0)(x)| \leqslant \phi(x)$ uniformly in $x\in \bar{\Omega}$ and for $w_0$ in bounded subsets of $H^{1}_{0}(\Omega)$. Also, if $|w_0(x)|\leqslant \phi(x)$ for all $x\in \bar{\Omega}$,  then $|u(t+\sigma,\sigma,w_0)(x)|\leqslant \phi(x)$ for all $x\in \bar{\Omega}$ and $t \geqslant 0$.
\end{itemize}
\end{prop}
\begin{proof}
Let B be a bounded set in $H^{1}_{0}(\Omega)$ and assume that $w_0\in B.$
\begin{itemize}
\item[(i)] Let $w^+(t+\sigma,\sigma,|w_0|)$ be the solution of 
\begin{equation}
\left\{ \begin{array}{lcc}
w_t + Aw(t)= -\nu C_0w+C_1 \;\;\; x\in \Omega, \;t>\sigma,\\\\
w=0, \;\; in\;\;x\in \partial \Omega, \;t>\sigma,\\\\
w(\sigma,|w_0|)=|w_0|.
\end{array}
\right.
\end{equation}

Since $g(t,w)\leqslant -\nu C_0w+C_1$ we obtain by Corollary \ref{corolariocomparacion}   $u(t+\sigma,\sigma,w_0)\leqslant w^+(t+\sigma,\sigma,|w_0|)$ as long as $u(t+\sigma,\sigma,w_0)$ exists. We consider 
\begin{equation}
\left\{ \begin{array}{lcc}
w_t + Aw(t)= -\nu C_0w-C_1 \;\;\; x\in \Omega, \;t>\sigma,\\\\
w=0, \;\; in\;\;x\in \partial \Omega, \;t>\sigma,\\\\
w(\sigma,|w_0|)=|w_0|,
\end{array}
\right.
\end{equation}
to obtain  $u(t+\sigma,\sigma,w_0) \geqslant w^-(t+\sigma,\sigma,-|w_0|)=-w^+(t+\sigma,\sigma,|w_0|)$ as long as $u(t+\sigma,\sigma,w_0)$ exists. Therefore
\begin{equation}\label{modulodeu}
|u(t+\sigma,\sigma,w_0)| \leqslant w^+(t+\sigma,\sigma,|w_0|),
\end{equation}
as long as $u(t+\sigma,\sigma,w_0)$ exists.
Let $\phi \in L^{\infty}(\Omega)$ be the solution of \eqref{A+C_0}, define $v=w^+-\phi,$ which satisfies a linear homogeneous equation $v_t+(A+\nu C_0)v=0$.

From Lemma \ref{Lema12.7}, 
\begin{equation}
\|v(t) \|_{L^{\infty}(\Omega)} =\| e^{-(A+\nu C_0)t}( |w_0| - \phi) \|  \leqslant \Tilde{M} t^{-\delta_0} e^{-\lambda_1 t}\| |w_0| - \phi \|_{L^q(\Omega)}
\end{equation}
and
$$
\|w^+(t)\|_{L^{\infty}(\Omega)} \leqslant \|v(t)\|_{L^{\infty}(\Omega)} + \|\phi\|_{L^{\infty}(\Omega)}\;\;\forall t>0.
$$
Consequently the $L^{\infty}(\Omega)$ norm of $u(t+\sigma,\sigma,w_0)$ remains bounded uniformly in time as long as it exists. Using  the variation of constants formula for the solution we have 
\begin{equation}
       \begin{array}{rl}
         \|u(t+\sigma,\sigma,w_0)\|_{H^{1}_{0}} & \leqslant \|e^{-At}w_0 \|_{H^{1}_{0}} + \displaystyle \int_\sigma^{t+\sigma} \|A^{1/2}
         e^{-A(t+\sigma-s)}\| \|g(u(s+\sigma,\sigma,w_0))\|ds \\
         & \leqslant  M\|w_0\|_{H^{1}_{0}} + MC\int_0^t(t-\theta)^{-1/2}d\theta.
       \end{array}
\end{equation}
Thus the $H^{1}_{0}(\Omega)$ norm of $u(t+\sigma,\sigma,w_0)$ remains bounded on finite time intervals and therefore the solution is global.
\item[(ii)]Since  assumption (\textbf{D}) holds, using Lemma \ref{Lema12.7}, $v=w^+-\phi$ satisfies 
$$
\| (v(t)\|_{L^{\infty}(\Omega)}\leqslant \tilde{M}e^{-\lambda_1t} t^{-\delta_0}\| |w_0|-\phi\|_{L^{q}(\Omega)},
$$
for all $t>0.$ Consequently,
$$
\|w^+(t+\sigma,\sigma,|w_0|)\|_{L^{\infty}(\Omega)}\leqslant \| (v(t)\|_{L^{\infty}(\Omega)} + \|\phi\|_{L^{\infty}(\Omega)}\longrightarrow \|\phi\|_{L^{\infty}(\Omega)}
$$
as $t \to \infty$, uniformly for $w_0\in B,$  then, using \eqref{modulodeu} we obtain \eqref{k_infty}.\\
Now we are going to use the variation of constants formula for the proof of \eqref{k_r}. Let $t_0$ be such that $\|u(t+\sigma,\sigma,w_0)\|_{L^{\infty}(\Omega)} \leqslant K_\infty +1$ for $t \geqslant t_0$. Thus 
\[
\begin{array}{rl}
\|u(t+\sigma,\sigma,w_0)\|_{\alpha}& =  \| e^{-(A+\nu C_0I)t}w_0+  \displaystyle \int_\sigma^{t+\sigma} e^{-(A+\nu C_0I)(t+\sigma-s)}g(u(s))ds\|_{\alpha}\\\\
&\leqslant Me^{-\lambda_1t}t^{-\alpha}\|w_0\|_{X}+MC(K_\infty) \displaystyle \int_0^t e^{-\lambda_1(t-\theta)}(t-\theta)^{-\alpha}d\theta,
\end{array}
\]
then $\displaystyle\limsup_{t\longrightarrow \infty} \|u(t+\sigma,\sigma,w_0)\|_{\alpha} \leqslant M_1C(K_\infty) \displaystyle \int_0^\infty e^{-\lambda_1 s} s^{-\alpha}ds$, for $\alpha\in [\frac12,1)$.

Therefore, we get that there exists a constant $K_\alpha$ such that \eqref{k_r} holds. 
\item[(iii)]Let $\phi$ be the solution of \eqref{A+C_0}. We know that $w^+=v+\phi$ and since $v$ goes to zero uniformly in $x \in \bar{\Omega}$ and for $w_0\in B$, then using \eqref{modulodeu} we have 
$$
\limsup_{t\longrightarrow \infty}|u(t+\sigma,\sigma,x)|\leqslant \phi(x) \;\;\; \textrm{for}\;\textrm{all}\;x \in \Omega.
$$
Now, if $|w_0| \leqslant \phi$ for all $x \in \bar{\Omega}$, since $|w_0|=w(\sigma,|w_0|) \leqslant v + \phi$, follows that  $v(t)\leqslant 0 $ on $\Omega$ and then $v(t)(x)\leqslant 0$. Thus $|u(t+\sigma,\sigma,w_0)| \leqslant \phi(x)$ for all $x \in \Omega$ and for all $t \geqslant 0.$
\end{itemize}
\end{proof}

\begin{corol}
Let the assumptions of Corollary \ref{corolariocomparacion} be satisfied. If (\textbf{S}) holds, then the solutions of \eqref{problema} with $w_0 \in H^{1}_{0}(\Omega)$ are globally defined.
\end{corol}
\begin{proof}
Follows immediately from Proposition \ref{proposicion12.8} and  Corollary \ref{corolariocomparacion}.
\end{proof}
\begin{obs}
Notice that the previous corollary not only allows us to guarantee that the solutions of  \eqref{problemau} are globally defined, but also those of \eqref{problema}. From that, we also obtain bounds on the solutions of problem \eqref{problema} as well. With this in mind, we are ready to prove the existence of attractors for the autonomous and non-autonomous case.
    
\end{obs}

\section{Pullback Attractors and Multivalued Processes}

In this section we introduce multivalued processes in order to study the asymptotic behavior of the solutions of \eqref{problemau} with $h$ non-zero. The autonomous case, that is, the existence of a global attractor for a multivalued semiflow, follows  as a corollary.

Let $X$ be a metric space with metric $\rho:X\times X\to \R^+$ and $P(X)$ $\mathscr{B}(X)$ be  the set of subsets and bounded subsets of $X$, respectively.  

Let us denote $W_\tau = C([\tau,\infty), X)$ and let $\mathcal{G} = \{\mathcal{G}(\tau)\}_{t \in \R}$ consists of maps $\phi \in W_\tau$. Let us consider the following conditions:

 \begin{itemize}
     \item[(C1)] For each $\tau \in \R$ and $z \in X$ there exists at least one $\phi \in \mathcal{G}(\tau) $ with $\phi(\tau)=z$;
     \item[(C2)] If $\phi \in \mathcal{G}(\tau)$ and $s \geqslant 0$, then $\phi^{+s} \in \mathcal{G}(\tau + s)$, where $\phi^{+s}:= \phi\big|_{[\tau + s , \infty)}$;
 \item[(C3)](Concatenation) If $\phi,\psi \in \mathcal{G}$ with $\phi \in \mathcal{G}(\tau), \psi \in\mathcal{G}(r)$ and $\phi(s)=\psi(s)$ for some $s \geqslant r \geqslant \tau$, then $\theta \in \mathcal{G}(\tau)$, where 
 $\theta (t): = \left\{ \begin{array}{lc} \phi(t), & t \in [\tau,s] ; \\  \psi (t), &  t \in (s, \infty).
\end{array} \right. $   
\item[(C4)] If $\{\phi_j\}_{j \in \N} \subset \mathcal{G}(\tau)$ and $\phi_j(\tau) \rightarrow z$, then there exists a subsequence $\{\phi_{\mu}\}_{\mu \in \N}$ of $\{\phi_j\}_{j \in \N}$ and $\phi \in \mathcal{G}(\tau)$ with $\phi(\tau)=z$ such that $\phi_{\mu}(t) \rightarrow \phi(t)$ for each $t \geqslant \tau.$

     \end{itemize} 
     
We define a multivalued map $U(t,\tau):X \rightarrow P(X)$ associated with the family $\mathcal{G}$ in the following way:
$$
U(t,\tau)x:= \{\phi(t); ~\phi \in \mathcal{G}(\tau), ~\phi(\tau) =x\}.
 $$
If (C1)-(C2) hold, then the map $U$ is a multivalued process, that is:
\begin{itemize}
    \item $U(t,t)x=x$ for all $t \in \R,~x\in X;$
    \item $U(t,\tau)x \subset U(t,s)U(s,\tau)x$ for all $- \infty < \tau \leqslant s \leqslant t < \infty,~x \in X;$ where $U(t,s)C = \bigcup_{x \in C} U(t,s)x$ for any $C \subset X.$ 
\end{itemize}

\begin{prop}\label{procesogeneralizado} Assume that $f:\R\to \R$ and $a,h:\R\to\R, ~l:H^1_0(0,1) \to \R$,  are continuous and that condition $\mathbf{(S)}$ is satisfied, and let $g:\R\times H^1_0(0,1) \times C([0,\infty),H^1_0(0,1)) \to H^1_0(0,1)$ be the operator function defined by the map $g$ in \eqref{gpmeg}. The family $\mathcal{G}= \{\mathcal{G}(\tau)\}_{\tau \in \R},$ given by \\
$$
\mathcal{G}(\tau):=\left\lbrace v \!\in\! C([\tau,+\infty),H^1_0(\Omega))\!:\! v(t)\!=\!T(t-\tau)v_0 \!+\!\! \displaystyle \int_\tau^t \!\!T(t-s)g(s,v^s(s),v^s(\cdot))ds, v_0 \in H^1_0(\Omega)
\right\rbrace,
$$ 
 satisfies (C1)-(C4).
\end{prop}

\begin{proof}
It is easy to see that (C1), (C2) and (C3) hold. And (C4) follows in the following way
If $\phi_j \in \mathcal{G}(\tau), j \in \N ,~ 0 \leqslant t \leqslant t_1 $ such that $\phi_j(\tau)=v_j \rightarrow v_0$ and $0<\beta+1/2<1$ we have that 

\begin{equation*}
\begin{split}
\| A^{\beta}\phi_j(t) \|_{\frac12} & \leqslant \| A^{\beta}T(t-\tau)\phi_j(\tau) \|_{\frac12} + \displaystyle \int_{\tau}^{t} \| A^{\beta} T(t-s) g(s,\phi_j^s(s),\phi^{s}_j(\cdot))\|_{\frac12}ds  \\\\
& \leqslant M_{\beta} (t-\tau)^{-\beta}\|\phi_j(\tau)\|_{\frac12} + NM \displaystyle \int_\tau^t (t-s)^{-1/2 -\beta}ds  < \infty
\end{split}
\end{equation*}
since $X^{1/2 + \beta } \subset \subset H^{1}_{0}(\Omega)$, it follows that 
  $\{u_j(t): \; u_j(\cdot) \in  \mathcal{G}(\tau), \; j \in \N \}$ is precompact in $H^{1}_{0}(\Omega)$.
\begin{equation*}
\begin{split}
\|\phi_j(t_2)-\phi_j(t_1)\|_{\frac{1}{2}} & \leqslant \left\|(T(t_2-\tau)-T(t_1-\tau)) \phi_j(\tau)+\int_\tau^{t_2} T(t_2-s) g((s, \phi_j^s(s), \phi_j^s(\cdot)) d s\right. \\
& \left.-\int_\tau^{t_1} T(t_1-s) g((s, \phi_j^s(s), \phi_j^s(\cdot)) d s \pm \int_\tau^{t_1} T(t_2-s) g((s, \phi_j^s(s), \phi_j^s(\cdot)) d s \right\|_{\frac{1}{2}} \\
& \leqslant\|(T(t_2\!-\!t_1)\!-\!I) T(t_1-\tau) \phi_j(\tau)\|_{\frac{1}{2}}\!+\!N \!\int_\tau^{t_1}\!\!\!\|T(t_2-s)\!-\!T(t_1-s)\|_{L(X, X^{\frac{1}{2}})} d s \\
& +\int_{t_1}^{t_2}\|T(t_2-s) g((s, \phi_j^s(s), \phi_j^s(\cdot))\|_{\frac{1}{2}} d s \\
& \leqslant \|(T(t_2-t_1)-I)A^{-\beta} T(t_1-\tau) A^\beta\phi_j(\tau)\|_{\frac{1}{2}} \\
& +\!N\! \int_\tau^{t_1}\!\!\|(T(t_2\!-\!t_1)-I) A^{-\beta}\! A^{\beta+\frac{1}{2}} T(t_1-s)\|_{L(X)} d s\!+\!N M^{\prime} \int_{t_1}^{t_2}(t_2-s)^{-\frac{1}{2}} d s \\
& \leqslant C(t_2-t_1)^{\beta}  +N K \int_\tau^{t_1}(t_2-t_1)^\beta(t_1-s)^{-\frac{1}{2}-\beta} d s+N K^{\prime} \int_{t_1}^{t_2}(t_2-s)^{-\frac{1}{2}} d s.
\end{split}
\end{equation*}
This last expression is independent of $\phi_j, \; j \in \N$, and tends to zero as $t_2 - t_1 \rightarrow 0$  since $T(t)$ is strongly continuous at $t=0$.
Note also that, for each $T>0$,
\[
\sup_{t \in [\tau,\tau+T], \; j \in \N} \|\phi_j(t)\|_{H^{1}_{0}}  < C.
\]
Then, (C4) follows. 

 \end{proof}
\begin{defin}\cite{simsen2016}
    The map $\psi: \R \rightarrow X$ is called a global solution through $x \in X $ if $\psi(\tau)= x$ for some $\tau \in \R$ and $\psi_{s}= \psi\big|_{[\tau + s , \infty)} \in \mathcal{G}(\tau+s)$ for all $s \in \R.$
    
\end{defin}

Notice that every global solution $\psi$ satisfies 
$$
\psi(t) \in U(t,s)\psi(s), ~for ~any~s\leqslant t.
 $$
\begin{defin}\cite{simsen2016}
Let $A = \{A(t)\}_{t \in \R}$ be a family of subsets of $X$.
\begin{itemize}
\item $A$ is positively invariant if $U(t, \tau )A(\tau ) \subset A(t)$ for all $ - \infty < \tau \leqslant t < \infty$;
\item  $A$ is negatively invariant if $A(t) \subset U(t,\tau )A(\tau)$ for all $ - \infty < \tau \leqslant t < \infty$;
\item $A$ is invariant if $U(t, \tau )A(\tau ) = A(t) $ for all $ -\infty < \tau \leqslant t < \infty.$

\end{itemize}
    
\end{defin}

\begin{defin}\cite{simsen2016}
    The family $K = \{K(t)\}_{t \in R}$ is called pullback attracting for $U$ if
$$
dist(U(t, s) B,K(t)) \rightarrow 0, ~as ~ s \rightarrow - \infty, ~for~ all ~B \in B(X) ,~ t \in \R.
$$
That is, if it pullback attracts any bounded set at any time $t \in \R$.
\end{defin}

\begin{defin}\cite{simsen2016}
The family $\mathcal{B} = \{\mathcal{B}(t)\}_{t \in \R}$ is said to be a global pullback attractor for $U$ if:
\begin{itemize}
    \item[(1)] $\mathcal{B}$ is compact;
    \item [(2)] $\mathcal{B}$ is pullback attracting;
    \item[(3)] $\mathcal{B}$ is negatively invariant;
    \item[(4)] $\mathcal{B}$ is minimal, that is, if $\tilde{\mathcal{B}}= \{\tilde{\mathcal{B}}(t)\} $ is a closed pullback attracting family, then $\mathcal{B}(t) \subset \tilde{\mathcal{B}}(t)$ for all $ t \in \R.$
    
    \end{itemize}

\end{defin}
An essential property for proving the existence of a pullback attractor is the so-called pullback asymptotic compactness.

\begin{defin}\cite{simsen2016}
$U$ is called pullback asymptotically compact at time $t$ if for all
$B \in B(X)$ each sequence $\{\xi_j\}_{j \in \N}$ such that $\xi_j \in U(t, \tau_j) B$, where $\tau_j \rightarrow - \infty $, has a convergent subsequence. If this property is satisfied for each time $t \in \R$, then we say that $\mathcal{G}$ is pullback asymptotically compact.
\end{defin}

\begin{defin}\cite{simsen2016}
The family $\{A(t)\}_{t \in R}$ pullback absorbs bounded subsets of $X$ if for each $t \in \R,~ B \in B(X)$ there exists $T = T(t, B) \leqslant t$ such that $U(t,\tau)B \subset A(t),$ for 
all $\tau \leqslant T.$
\end{defin}

\begin{defin}\cite{simsen2016}
$U$ is called pullback bounded dissipative if there exists a family
$\mathcal{B}_0 := \{B_0(t)\}_{t \in \R}$ with $B_0(t) \in B(X)$ for any $t \in \R$ which pullback absorbs bounded subsets of $X$. $\mathcal{B}_0$ is said to be pullback absorbing. It is said to be monotonically pullback bounded dissipative if, in addition, $B_0(s) \subset B_0(t)$ for every $s \leqslant t$.
\end{defin}

The previous definition is the  strongly pullback bounded dissipative property introduced in \cite{caraballo2010existence} to study the non-autonomous single-valued case.
 
\begin{theo}\cite{simsen2016}\label{atractorpullback} Let (C1)-(C2), (C4) hold. Then $U$ is pullback asymptotically
compact and monotonically pullback bounded dissipative if and only if it possesses the unique backwards bounded global pullback attractor $\mathcal{B} = \{\mathcal{B}(s) : s \in \R \}$
defined by
$$
\mathcal{A}(t) = \bigcup_{B \in B(X)} \omega(t,B).
$$
If (C3) is also satisfied, then $\mathcal{A}$ is invariant.
\end{theo}

\begin{theo}
 Let $f, a\circ l, h$ be continuous functions such that $(\boldsymbol{D})$ holds. Assume also that $f$ is such that, for every $r>0$, there is a constant $\gamma=\gamma(r)>0$ such that $\gamma I+\lambda f(\cdot)$ is increasing in $[-r,r]$. Then the multivalued process $U(t, \tau)$ generated by $\mathcal{G}$ defined in Proposition \ref{procesogeneralizado} has a pullback attractor $\mathcal{B}=\{\mathcal{B}(s): s \in \mathbb{R}\}$.
\end{theo}
\begin{proof}
Properties $(C 1)-(C 4)$ have been checked in Proposition \ref{procesogeneralizado}. To prove that $\mathcal{G}$ is pullback asymptotically compact, pullback bounded dissipative and monotonically pullback bounded dissipative we apply the results of Proposition \ref{proposicion12.8} and the compactness of the embedding of $X^{\frac{1}{2}+\beta}$ into $X^{\frac{1}{2}}=H_0^1(0,1)$ for $0<\beta<\frac{1}{2}$. That is, defining $B_0(t):=\overline{B_{K^\alpha}^{H_0^1}(0)}$ for all $t \in \mathbb{R}$ follow by Proposition \ref{proposicion12.8} and by Corollary \ref{corolariocomparacion}  , that $U$ is monotonically pullback bounded dissipative and the result follows using Theorem \ref{atractorpullback}.
 \end{proof}

 \begin{corol}
 Let $f, a\circ l, h$ be continuous functions such that $(\boldsymbol{D})$ holds. Assume also that $f$ is such that, for every $r>0$, there is a constant $\gamma=\gamma(r)>0$ such that $\gamma I+\lambda f(\cdot)$ is increasing in $[-r,r]$. Then $\mathscr{A}^{-}, \mathscr{B}(t), \mathscr{A}^{+}$are subsets of
$$
\Sigma(\phi):=\left\{u \in L^{\infty}(\Omega),\left|u\left(t,\sigma, u_0\right)(x)\right| \leqslant \phi(x),t\geqslant \sigma \ \text { a.e. } x \in \Omega\right\} .
$$
for all $t \in \mathbb{R}$, where $\mathscr{A}^{-}\left(\right.$resp. $\left.\mathscr{A}^{+}\right)$is the global attractor of multivalued semiflow $\mathcal{G}^{-}\left(\right.$resp. $\left.\mathcal{G}^{+}\right)$ defined by the solutions of \eqref{f-} (resp. \eqref{f+}) and $\mathscr{B}=\{\mathscr{B}(t): t \in \mathbb{R}\}$ is the multivalued pullback attractor of the multivalued process $U(t, \tau)$ associated to the solutions of \eqref{problematimereparametrizado} .
     \end{corol}

\begin{proof}
    The existence of the attractors is guaranteed by Theorem \ref{atractorpullback} and its autonomous version, when $h=0$, and the comparison between them follows using Corollary \ref{corolariocomparacion} and the uniform bounds follow as in Theorem \ref{atractorpullback}. 

\end{proof}
  
\section*{Data availability statement}
The data that support the findings of this study are available within the article.

\section*{Conflict of interest}
The authors declare that they have no conflict of interest.

\bibliographystyle{acm}

\end{document}